\documentclass{amsart}
\usepackage{amsthm, amssymb}

\newtheorem{theorem}{Theorem}
\newtheorem{proposition}[theorem]{Proposition}

\DeclareMathOperator*{\init}{init}
\DeclareMathOperator*{\rk}{rk}
\newcommand{\A}{\mathcal{A}}
\newcommand{\upset}{\mathclose\uparrow}

\newcommand{\newword}[1]{\textbf{#1}}

\let\oldmarginpar\marginpar
\renewcommand\marginpar[1]{\-\oldmarginpar[\footnotesize #1]{\raggedright\footnotesize #1}}

\title{A Gr\"obner basis for the graph of the reciprocal plane}
\author{Alex Fink}
\author{David E Speyer}
\author{Alexander Woo}

\begin{document}
\begin{abstract}
Given the complement of a hyperplane arrangement,
let $\Gamma$ be the closure of the graph of the map inverting each of its defining linear forms.
The characteristic polynomial manifests itself in the Hilbert series of~$\Gamma$
in two different-seeming ways, one due to Orlik and Terao and the other to Huh and Katz.  
We define an extension of the no broken circuit complex of a matroid
and use it to give a direct Gr\"obner basis argument
that the polynomials extracted from the Hilbert series in these two ways agree.
\end{abstract}
\maketitle

Let $\mathcal{A}$ be an arrangement of $n+1$ distinct hyperplanes
in an $r+1$-dimensional vector space $L$ over some field $\Bbbk$.
Assume that $\mathcal{A}$ is essential, that is, that the intersection of all its hyperplanes is~$\{0\}$.
We coordinatize $\mathcal{A}$ by fixing a linear form $x_i$ vanishing on the $i$\/th hyperplane.
These linear forms provide an injective linear map $L \to \Bbbk^{n+1}$, and we will identify $L$ with its image in $\Bbbk^{n+1}$ from now on.
The hyperplane arrangement $\mathcal{A}$ is then $\{L\cap (x_i=0) \mid 0\leq i\leq n\}$.  
The \newword{arrangement complement} $\widetilde{L} = L\setminus \mathcal{A}$ is the complement of the coordinate hyperplanes.  
We can projectivize $\widetilde{L}\subseteq\Bbbk^{n+1}$ to form
$\mathbb{P}(\widetilde{L})\subseteq\mathbb{P}^n$.

Naturally associated to $\mathcal A$ is a matroid $M$
of rank $r+1$ represented over~$\Bbbk$ on the set of the $n+1$ hyperplanes in~$\mathcal{A}$,
with no loops and no collinear points.  
The matroid $M$ encodes the dependencies among the coordinate functions $x_i$ on~$L$\/:
a~set of coordinate functions $\{x_i \mid i\in\mathcal{C}\}$ is linearly dependent if and only if $\mathcal{C}$ is a dependent set of~$M$.  

An important invariant associated to any matroid,
and via this to any hyperplane arrangement,
is the \newword{characteristic polynomial} $\chi_M(q)$, defined as
$$\chi_M(q)=\sum_{F\subseteq M}\mu(\hat0, F)\, q^{r+1-r(F)},$$
where the sum is over the flats of the matroid $M$.
Here $\mu$ denotes the M\"obius function, and $r(F)$ is the rank of the flat $F$.

The Cremona transformation $\mathbb{P}^n\dashrightarrow \mathbb{P}^n$ is defined by sending $(z_0:\cdots:z_n)$ to $(z_0^{-1}:\cdots:z_n^{-1})$.  The \newword{reciprocal plane} $R_\A$ is the closure of the image of the Cremona transformation restricted to $\mathbb{P}(\widetilde{L})$, embedded as a closed subvariety of $\mathbb{P}^n$.  The \newword{reciprocal graph} $\Gamma_\A$ is the closure of the graph of the Cremona transformation restricted to $\mathbb{P}(\widetilde{L})$, embedded as a closed subvariety of $\mathbb{P}^n\times\mathbb{P}^n$.  Note that $R_\A=\pi_2(\Gamma_\A)$, where $\pi_2$ denotes projection to the second factor.

In previous work, two ways have appeared
of recovering the characteristic polynomial from the above geometry.
Orlik and Terao~\cite{OT} show that the Hilbert series of (the projective coordinate ring of) $R_\A$ is given by
$$H(R_\A;t) = \sum_{i=0}^{r+1} w_i \left(\frac{t}{1-t}\right)^i,$$
where $(-1)^i w_i$ is the coefficient of $q^{r+1-i}$ in $\chi_M(q)$;
note that $\chi_M(q)$ is a polynomial in~$q$ of degree~$r+1$ whose coefficients alternate in sign.
In terms of the Grothendieck ring $K^0(\mathbb{P}^n)$, this means that
$$[R_\A]=\sum_{i=0}^r w_{i+1}[(\Bbbk^*)^{i}],$$
where by $[(\Bbbk^*)^a]$ we mean
$$[(\Bbbk^*)^a]=\sum_{j=0}^a (-1)^j\binom{a+1}{j}[\mathbb{P}^{a-j}].$$
On the other hand, Huh and Katz~\cite{HK} show that the cohomology class of $\Gamma_\A$ in $H^{2(2n-r)}(\mathbb{P}^n\times\mathbb{P}^n)$ is given by
$$[\Gamma_\A]=\sum_{i=0}^r \overline{w}_i[\mathbb{P}^{r-i}\times\mathbb{P}^i],$$ 
where $(-1)^i\overline{w}_i$
is the coefficient of $q^{r-i}$ 
in the \newword{reduced characteristic polynomial} $\overline{\chi}_M(q) := \chi_M(q)/(q-1)$.  This is a rather surprising coincidence!  By the Chern map, the class of a subvariety $Y\subseteq X$ in $H^*(X)$ 
can be thought of as the leading terms of its class in $K^0(X)$.  
Hence this coincidence between a $K$-class and a cohomology class
suggests a correspondence between
the {\em leading} terms of $[\Gamma_\A]\in K^0(\mathbb{P}^n\times\mathbb{P}^n)$ 
and {\em all} the terms of $[R_\A]=[\pi_2(\Gamma_\A)]\in K^0(\mathbb{P}^n)$.
Note that this relationship is not simply the one arising from pushforward along the projection,
which cannot be computed from only the leading terms of $[\Gamma_\A]$.

Proudfoot and the second author~\cite{PS} give an explanation of the Orlik--Terao result in terms of combinatorial commutative algebra by showing that $R_\A$ has a Gr\"obner degeneration to the Stanley--Reisner scheme associated to the \newword{no broken circuit complex} $\Delta^{\mathrm{NBC}}_M$, whose faces are counted by the characteristic polynomial.  The no broken circuit complex is a cone over the vertex corresponding to the first variable, so one can define the \newword{reduced no broken circuit complex} $\Delta^{\mathrm{RNBC}}_M$ by deleting the cone point.  The faces of $\Delta^{\mathrm{RNBC}}_M$ are counted by the reduced characteristic polynomial.
(The authors caution the reader that sources differ as to whether $\Delta^{\mathrm{NBC}}_M$ should be called the ``no broken circuit complex" or the ``broken circuit complex".)

In this paper, we give a similar combinatorial commutative algebra explanation for the Huh--Katz
result by defining a family of \newword{extended no broken circuit complexes}
$\Delta^{\mathrm{ENBC}}_{M,\prec}$, one for each total order $\prec$ on $\{0,\ldots,n\}$,
and showing that $\Gamma_\A$ has Gr\"obner degenerations to the Stanley--Reisner schemes
of $\Delta^{\mathrm{ENBC}}_{M,\prec}$.  Our simplicial complexes
$\Delta^{\mathrm{ENBC}}_{M,\prec}$ are all pure with one facet for every face
of $\Delta^{\mathrm{RNBC}}_M$, 
explaining the common appearance of the (reduced) characteristic polynomial in these two different settings.

By counting the faces of $\Delta^{\mathrm{ENBC}}_{M,\prec}$, we also obtain the bigraded Hilbert series for $\Gamma_\A$, which is
\begin{align*}
H(\Gamma_\A; q, t)
&=\sum_{i=0}^r \overline{w}_i \left(1+\frac{q}{1-q}\right)^{r+1-i}\left(\frac{t}{1-t}\right)^i\left(1+\frac{t}{1-t}\right) \\
&= \frac{1}{1-qt}\left(\frac{t}{t-1}\right)^{r+1}\chi_M\left(\frac{t-1}{(1-q)t}\right).
\end{align*}

The coordinate ring of $\Gamma_\A$ was recently independently studied
by Garrousian, Simis, and Tohaneanu~\cite{GST}.  We recover their results on a presentation
for the coordinate ring of $\Gamma_\A$.  They furthermore show that $\Gamma_\A$ is
arithmetically Cohen--Macaulay.  It would be interesting to recover this result by showing that
our complexes $\Delta^{\mathrm{ENBC}}_{M,\prec}$ are shellable.  

\section*{Acknowledgments}
This paper arose from a working group at the Fields Institute thematic semester
\emph{Combinatorial algebraic geometry}, which included 
Laura Escobar, Federico Galetto, Christian Hasse, Alexandra Seceleanu, and Kristin Shaw.
We thank the working group for fruitful discussions,
Diane Maclagan and Greg Smith for organizing the thematic semester,
and the Fields Institute staff for their hospitality.

The first author was supported by EPSRC grant EP/M01245X/1; the second author was supported by NSF grant DMS-1600223;
the third author was supported by Simons Collaboration Grant 359792.

\section{The extended no broken circuit complex}\label{sec:ENBC}

Given a matroid $M$ of rank $r+1$ on a ground set $\{0,\ldots,n\}$, 
a \newword{circuit} is a minimal dependent set, 
and a \newword{broken circuit} is the result of deleting the {\em least} element from any circuit
in the natural order $0<1<\cdots<n$.  
The \newword{no broken circuit complex} $\Delta^{\mathrm{NBC}}_M$ is the simplicial complex whose minimal {\em nonfaces} are the broken circuits of $M$.  In other words, the facets of $\Delta^{\mathrm{NBC}}_M$ are the maximal sets {\em not} containing a broken circuit.  Note that $0$ is never in a broken circuit, and hence the vertex $0$ is always a cone point of $\Delta^{\mathrm{NBC}}_M$.  Following Brylawski~\cite{B}, we define the \newword{reduced no broken circuit complex} $\Delta^{\mathrm{RNBC}}_M$ as the simplicial complex obtained by deleting the vertex $0$ from $\Delta^{\mathrm{NBC}}_M$.

It is a classical result, due to Whitney~\cite{W} for graphical matroids and Brylawski~\cite{B} in general,
that the characteristic polynomial of $M$ is given by 
$$\chi_M(q)=\sum_{i=0}^{r+1} (-1)^i w_i\, q^{r+1-i},$$
where $w_i$ (which is sometimes known as the $i$-th \newword{Whitney number of the first kind}) is the number of faces of $\Delta^{\mathrm{NBC}}_M$ with $i$ vertices.  
Note that the reduced characteristic polynomial satisfies
$$\overline{\chi}_M(q)=\chi_M(q)/(q-1)=\sum_{i=0}^r (-1)^i \overline{w}_i q^{r-i},$$ 
where $\overline{w}_i$ is the number of faces of $\Delta^{\mathrm{RNBC}}_M$ with $i$ vertices.

Given a second partial order $\prec$ on the set $\{0,\ldots,n\}$,
we define the \newword{extended no broken circuit complex} $\Delta^{\mathrm{ENBC}}_{M,\prec}$ as follows.  
The complex $\Delta^{\mathrm{ENBC}}_{M,\prec}$ has $2(n+1)$ vertices, which we denote $\{x_0,\ldots,x_n,y_0, \ldots,y_n\}$.  
Given any face $F\in\Delta^{\mathrm{RNBC}}_M$ (including $F=\emptyset$), 
let $L(F)$ be the basis of~$M$ containing~$F$ which is lexicographically maximal with respect to~$\prec$.  
To be precise,
let $i\upset=\{j\mid i\prec j\}$ be the up-set generated by $i$ in~$\prec$.
Then $i\in L(F)$ if $i\in F$ or 
$$\rk(i\upset\cup F)> \rk((i\upset\setminus\{i\})\cup F).$$  
Given any $F\in\Delta^{\mathrm{RNBC}}_M$, define 
$$\overline{F}=\{y_i\mid y\in F\} \cup \{x_j\mid j\in L(F)\setminus F\}\cup\{y_0\}.$$
Thus, putting aside $y_0$ which will be a cone point,
$L(F)$ is the set of subscripts of vertices of~$\overline{F}$,
with $F$ distinguished within it as the set of subscripts of $y$-variables.
The facets of $\Delta^{\mathrm{ENBC}}_{M,\prec}$ are the sets $\overline{F}$; in other words,
$$\Delta^{\mathrm{ENBC}}_{M,\prec}=\{G\mid G\subseteq \overline{F}, F\in\Delta^{\mathrm{RNBC}}_M\}.$$

\section{Squarefree initial ideals and Stanley--Reisner rings}

Let $S=\Bbbk[x_0,\ldots,x_n, y_0,\ldots,y_n]$ be the bihomogeneous coordinate ring of $\mathbb{P}^n\times\mathbb{P}^n$, and fix a total term order on $S$.  Given $f\in S$, $\init f$ is the largest monomial in the support of $f$.  
Given an ideal $I\subseteq S$, its initial ideal is $\init I=\{\init f\mid f\in I\}$.  
Since our term order is a total order, $\init I$ is a monomial ideal.  It is a general fact that there exists a flat degeneration from the scheme $V(I)$ to $V(\init I)$ preserving the Hilbert series and hence the cohomology class.

Let $\Delta$ be a simplicial complex with vertex set $\{x_0,\ldots,x_n,y_0,\ldots,y_n\}$.  The Stanley--Reisner ideal $I(\Delta)\subseteq S$ is the squarefree monomial ideal generated by $\prod_{v\in A} v$ for all subsets $A$ such that $A\not\subseteq \Delta$.  The nonzero monomials of $S/I(\Delta)$ are precisely those whose variables are faces of $\Delta$, so the bigraded Hilbert series 
with $x$-degree counted by $q$ and $y$-degree counted by $t$ is given by
\begin{equation}\label{eq:H(q,t)}
H(S/I(\Delta);q,t)=\sum_{F\in\Delta} \left(\frac{q}{1-q}\right)^{x(F)} \left(\frac{t}{1-t}\right)^{y(F)},
\end{equation}
where the sum is over all faces $F$ in $\Delta$ and $x(F)$ and $y(F)$ denote respectively
the number of $x$-vertices and $y$-vertices in $F$.

Since $I(\Delta)$ is squarefree, the subscheme $V(\Delta)$ of $\mathbb{P}^n\times\mathbb{P}^n$ defined by $I(\Delta)$ is reduced, and its irreducible components are the subspaces spanned by $v\in F$ as $F$ ranges over the {\em maximal} faces of $\Delta$.  Hence the class of $V(\Delta)$ in $H^*(\mathbb{P}^n\times\mathbb{P}^n)$  is 
\begin{equation}\label{eq:cohom class}
[V(\Delta)]=\sum_F [\mathbb{P}^{x(F)-1}\times\mathbb{P}^{y(F)-1}],
\end{equation}
where the sum is now over all faces $F$ of maximal dimension (which in general may not be all maximal faces).

\section{Main theorem and proof}

Our main theorem is the following.

\begin{theorem}
Let $\Gamma_\A$ be the reciprocal graph.
Fix a total order $\prec$ on~$\{0,\ldots,n\}$.
Let $>$ be a 
term order on $S$ 
such that $x_i>x_j$ if $i\prec j$ while $y_i>y_j$ if $i>j$,
and any term that is a multiple of~$y_0^{\raisebox{0.25ex}{\scriptsize a}}$ 
is less than any term of the same degree that is not.  Then
\begin{enumerate}
\item (Also~\cite[Thm.\ 4.2]{GST}) The defining ideal $I(\Gamma_\A)\subseteq S$ is generated by the following elements:
\begin{itemize}
\item $\sum_{i\in \mathcal{C}} a_i x_i,$ where $\mathcal{C}$ is a circuit and the $a_i\in\Bbbk$ define the relation given by the circuit.
\item $\sum_{i\in \mathcal{C}} a_i\prod_{j\in\mathcal{C}\setminus\{i\}} y_j$, where $\mathcal{C}$ is a circuit and $a_i$ are as above.
\item $x_iy_i-x_0y_0$, for all $i$ with $1\leq i\leq n$
\end{itemize}
\item The initial ideal $\init I(\Gamma_\A)\subseteq S$ is generated by the following elements:
\begin{itemize}
\item $\prod_{i\in \mathcal{B}} y_i$, where $\mathcal{B}$ is a broken circuit
\item $x_j\prod_{i\in I} y_i$, where $I$ is any subset of $\{1,\ldots, n\}$ (so {\em not} including $0$) and
$\rk(j\upset\cup I)=\rk((j\upset\setminus\{j\})\cup I)$.  (Note this includes the degenerate case where $j\in I$.)
\end{itemize}
\item The initial ideal $\init I(\Gamma_\A)$ is the Stanley--Reisner ideal $I(\Delta^{\mathrm{ENBC}}_{M,\prec})$ of the extended no broken circuit complex.
\end{enumerate}
\end{theorem}

\begin{proof}
Let $J$ be the ideal generated by the elements listed in part (1),
and let $K$ be the ideal generated by the elements of part (2).  We will show that $J\subseteq I(\Gamma_\A)$ and that $I(\Delta^{\mathrm{ENBC}}_{M,\prec})\subseteq K\subseteq \init J$.  We first explain how to conclude the proof using these facts.

Since $\Delta^{\mathrm{ENBC}}_{M,\prec}$ is pure-dimensional, $I(\Delta^{\mathrm{ENBC}}_{M,\prec})$ defines an equidimensional and reduced scheme.  Hence if $A$ is a monomial ideal containing $I(\Delta^{\mathrm{ENBC}}_{M,\prec})$ such that we have the equality $$[V(S/A)]=[V(\Delta^{\mathrm{ENBC}}_{M,\prec})]\in H^*(\mathbb{P}^n\times\mathbb{P}^n)$$ of cohomology classes (or equivalently of bidgrees), 
then $A=I(\Delta^{\mathrm{ENBC}}_{M,\prec})$~\cite[Exer.\ 8.13]{MS}.  
By construction, for each face $F\in\Delta^{\mathrm{RNBC}}_M$,
we have a facet $\overline{F}\in\Delta^{\mathrm{ENBC}}_{M,\prec}$,
and every facet has $r+2$ vertices.  Furthermore, $\overline{F}$ has
$r-|F|+1$ $x$-vertices and $|F|+1$ $y$-vertices.
Hence, by Equation~\eqref{eq:cohom class},
$$[V(\Delta^{\mathrm{ENBC}}_{M,\prec})]=\sum_{F\in\Delta^{\mathrm{RNBC}}_M} [\mathbb{P}^{r-|F|}\times\mathbb{P}^{|F|}],$$
where the sum is over all faces $F$ of $\Delta^{\mathrm{RNBC}}_M$.

On the other hand, Huh and Katz~\cite{HK} show that
$$[\Gamma_\A]=\sum_{i=0}^r \overline{w}_i[\mathbb{P}^{r-i}\times\mathbb{P}^{i}],$$  
where $(-1)^i\overline{w}_i$ is the coefficient of $q^{r-i}$ in $\overline{\chi}_M(q)$.  
Since $\overline{w}_i$ is the number of faces of $\Delta^{\mathrm{RNBC}}_M$ with $i$ vertices, $[\Gamma_\A]=[V(\Delta^{\mathrm{ENBC}}_{M,\prec})]$.  Taking an initial ideal preserves the cohomology class, so
$[V(\init I(\Gamma_\A))]=[V(\Delta^{\mathrm{ENBC}}_{M,\prec})]$, and $\init I(\Gamma_\A)=I(\Delta^{\mathrm{ENBC}}_{M,\prec})$.  Therefore, $$I(\Delta^{\mathrm{ENBC}}_{M,\prec})=K=\init J=\init I(\Gamma_\A),$$ and $J=I(\Gamma_\A)$.

To show $J\subseteq I(\Gamma_\A)$, we show the generators of each type vanish on $\Gamma_\A$.  
The generators involving only the $x$ variables come from the relations defining the linear space~$L$.  
On a point in $\widetilde{L}$ where $x_i\neq 0$ for all $i$, $$y_i=\frac{1}{x_i},$$ so $$x_iy_i=x_jy_j$$ for all $i$ and $j$, and in particular for $j=0$.  Also, given a relation $$\sum_{i\in\mathcal{C}} a_i x_i=0$$ on $L$ coming from a circuit $\mathcal{C}$, we have relations $$\sum_{i\in\mathcal{C}} \frac{a_i}{y_i} = 0,$$ or, clearing denominators, 
$$\sum_{i\in \mathcal{C}} a_i\prod_{j\in\mathcal{C}\setminus\{i\}} y_j = 0.$$

To show $K\subseteq \init J$, for each generator of $g\in K$, we find $f\in J$ such that $\init f=g$.  If $\mathcal{B}$ is a broken circuit, then $\mathcal{B}$ is some circuit $\mathcal{C}$ with its first element removed, and hence $\prod_{i\in \mathcal{B}} y_i$ is the leading term of $$\sum_{i\in \mathcal{C}} a_i\prod_{j\in\mathcal{C}\setminus\{i\}} y_j.$$

On the other hand, given a subset $I\subseteq \{1,\ldots,n\}$ and some element $j\in\{0,\ldots,n\}$ such that
$$\rk(j\upset\cup I)=\rk((j\upset\setminus\{j\})\cup I),$$
either we are in the degenerate case where $j\in I$, where $x_jy_j\in \init J$ since
$x_jy_j-x_0y_0\in J$, or there is some circuit $\mathcal{C}$ including $j$ and a subset of
$(j\upset\setminus\{j\})\cup I$.  
Let $$h=\sum_{k\in\mathcal{C}} a_k x_k\in J$$ be the relation given by $\mathcal{C}$.  Consider $$h'=\left(\prod_{i\in I} y_i\right)h\in J.$$  We can write
$$h'=a_j\left(\prod_{i\in I} y_i\right)x_j+\sum_{k\in \mathcal{C}\cap(j\upset\setminus\{j\})} a_k\left(\prod_{i\in I} y_i\right)x_k + \sum_{k\in \mathcal{C}\setminus j\upset} a_k\left(\prod_{i\in I} y_i\right) x_k.$$
Note that, if $k\in \mathcal{C}$ and $k\not\succ j$, then $k\in I$, so
$$h'=a_j\left(\prod_{i\in I} y_i\right)x_j+\sum_{k\in \mathcal{C}\cap (j\upset\setminus\{j\})} a_k\left(\prod_{i\in I} y_i\right)x_k + \sum_{k\in \mathcal{C}\setminus j\upset} a_k\left(\prod_{i\in I\setminus\{k\}} y_i\right) x_k y_k.$$

Since $x_iy_i-x_0y_0\in J$ for all $i$,
\begin{align*}
h'' &= h' + \sum_{k\in \mathcal{C}\setminus j\upset} a_k\left(\prod_{i\in I\setminus\{k\}} y_i\right) (x_0 y_0 - x_k y_k) \\
    &=
a_j\left(\prod_{i\in I} y_i\right)x_j+\sum_{k\in \mathcal{C}\cap (j\upset\setminus\{j\})} a_k\left(\prod_{i\in I} y_i\right)x_k \\
&\,\,\,\,+
\sum_{k\in \mathcal{C}\setminus j\upset} a_k\left(\prod_{i\in I\setminus\{k\}} y_i\right) x_k y_k +
\sum_{k\in \mathcal{C}\setminus j\upset} a_k\left(\prod_{i\in I\setminus\{k\}} y_i\right) (x_0 y_0 - x_k y_k) \\
&=a_j\left(\prod_{i\in I} y_i\right)x_j+\sum_{k\in \mathcal{C}\cap (j\upset\setminus\{j\})} a_k\left(\prod_{i\in I} y_i\right)x_k +
\sum_{k\in \mathcal{C}\setminus j\upset} a_k\left(\prod_{i\in I\setminus\{k\}} y_i\right) x_0 y_0\\
&\in J.
\end{align*}
The first term is the leading term, since it contains no $y_0$ and $x_j>x_k$ for all $k\succ j$, so
$$\init h''= x_j\prod_{i\in I} y_i\in K.$$

To show $I(\Delta^{\mathrm{ENBC}}_{M,\prec})\subseteq K$, suppose $$\prod_{a\in A}\prod_{b\in B} x_ay_b \not\in K.$$  Then $B$ does not contain a broken circuit, and for every $a\in A$,
$$\rk((B\setminus\{0\})\cup a\upset)>\rk((B\setminus\{0\})\cup(a\upset\setminus\{a\})).$$
Note $B\setminus\{0\}$ is a face $F\in\Delta^{\mathrm{RNBC}}_M$, and $A\subseteq L(F)$ by our condition on elements of~$A$, so 
$$\{x_a\mid a\in A\}\cup \{y_b\mid b\in B\}\subseteq L(F)\in\Delta^{\mathrm{ENBC}}_{M,\prec}.$$  Hence, 
\[\prod_{a\in A}\prod_{b\in B} x_ay_b \not\in I(\Delta^{\mathrm{ENBC}}_{M,\prec}).\qedhere\]

\end{proof}

\section{Hilbert series}

In this section, we state and prove our formula for the bigraded Hilbert series of $S/I(\Gamma_\A)$ and show how the Orlik--Terao and Huh--Katz results follow from this formula.

\begin{proposition}
The bigraded Hilbert series of $S/I(\Gamma_\A)$ is
\[H(\Gamma_\A;q,t) 
= \frac{1}{1-qt}\left(\frac{t}{t-1}\right)^{r+1}\chi_M\left(\frac{t-1}{(1-q)t}\right).\]
\end{proposition}

\begin{proof}
Because $I(\Delta^{\mathrm{ENBC}}_{M,\prec})$ is an initial degeneration of~$I(\Gamma_\A)$,
it has the same Hilbert function.  
Computing $H(\Delta^{\mathrm{ENBC}}_{M,\prec};q,t)$ is an enumerative problem, by~\eqref{eq:H(q,t)}.

We carry out this count by means of a partition of the faces of~$\Delta^{\mathrm{ENBC}}_{M,\prec}$.
To wit, given any face $F\in\Delta^{\mathrm{RNBC}}_M$, 
let
\[J(F)=\{G\in\Delta^{\mathrm{ENBC}}_{M,\prec} \mid G\cap \{y_1,\ldots,y_n\}=\{y_i\mid i\in F\}\}. \]
The subscripts of the $y$-vertices (not including $y_0$) of any facet of $\Delta^{\mathrm{ENBC}}_{M,\prec}$ make up a face of $\Delta^{\mathrm{RNBC}}_M$,
so the same is true for any face of $\Delta^{\mathrm{ENBC}}_{M,\prec}$.
Hence $\{J(F) \mid F\in\Delta^{\mathrm{RNBC}}_M\}$ gives a partition of $\Delta^{\mathrm{ENBC}}_{M,\prec}$.

Next we show that $J(F)$ is in fact the interval $[\{y_i\mid i\in F\}, \overline{F}]$, so that $G\in J(F)$ if and only if
\[\{y_i\mid i\in F\}\subseteq G\subseteq\overline{F}.\]
Since $\overline{F}\cap \{y_1,\ldots,y_n\} = F$, we have $[\{y_i\mid i\in F\}, \overline{F}]\subseteq J(F)$.
Now suppose we have a face $G\in J(F)\subseteq\Delta^{\mathrm{ENBC}}_{M,\prec}$.  Then $G\subset\overline{H}$
for some $H\sup-set F\in \Delta^{\mathrm{RNBC}}_M$.  However, if $F\subset H$, then $L(F)\setminus F\sup-set L(H)\setminus H$,
where $L(F)$ has the meaning it had in Section~\ref{sec:ENBC},
because if $j$ is independent of $H\cup(j\upset\setminus\{j\})$,
then $j$ is also independent of $F\cup(j\upset\setminus\{j\})$.
Hence, $G\subset\overline{F}$.

The contribution of $J(F)$ to the Hilbert function is
\begin{align}\label{eq:Hilb1}
&\phantom= \sum_{G\in J(F)} \left(\frac{q}{1-q}\right)^{x(G)} \left(\frac{t}{1-t}\right)^{y(G)}
\\&= \left(1 + \frac{q}{1-q}\right)^{|L(F)\setminus F|} \left(\frac{t}{1-t}\right)^{|F|} \left(1+\frac{t}{1-t}\right)\notag
\\&= \frac{t^r}{(1-q)(1-t)^{r+1}}\left(\frac{1-t}{(1-q)t}\right)^{r-|F|},\notag
\end{align}
since $|L(F)|=r+1$ for all $F$.
The Hilbert function is the sum of these contributions, 
and there are $\overline w_i$ faces $F$ of~$\Delta^{\mathrm{RNBC}}_M$ with $|F|=i$, giving
\begin{align*}
H(S/I(\Gamma_\A);q,t) 
&= \frac{t^r}{(1-q)(1-t)^{r+1}}\cdot\sum_{i=0}^r \overline w_i \left(\frac{1-t}{(1-q)t}\right)^{r-i}
\\&= \frac{(-t)^r}{(1-q)(1-t)^{r+1}}\cdot\overline\chi_M\left(\frac{t-1}{(1-q)t}\right)
\\&= \frac{1}{1-qt}\left(\frac{t}{t-1}\right)^{r+1}\chi_M\left(\frac{t-1}{(1-q)t}\right).\qedhere
\end{align*}
\end{proof}

Since $R_\mathcal A$ is the second projection of $\Gamma_\A$,
we may recover $H(R_\mathcal A;t)$ by evaluating $q$ at~0, 
corresponding to intersection with the subring $\Bbbk[y_0,\ldots,y_n]$ of~$S$.
This evaluation is
\[H(S/I(\Gamma_\A);0,t) 
= \left(\frac{t}{t-1}\right)^{r+1}\chi_M\left(\frac{t-1}{t}\right)\]
agreeing with the result of Orlik and Terao.

Agreement with the result of Huh and Katz,
invoking~\eqref{eq:cohom class},
was used in our proof.  Note, though,
that this cohomology class can also be calculated directly from the Hilbert series
using the method of multidegrees \cite[\S8.5]{MS}.

\section{Example: Braid arrangement for $A_3$}

We work out the details for the matroid of braid arrangement of $A_3$, also known as the graphical arrangement for the complete graph $K_4$.  This is a matroid of rank 3 on 6 elements, with characteristic polynomial $\chi_{K_4}(q)=q^3-6q^2+11q-6$ and reduced characteristic polynomial $\overline{\chi}_{K_4}(q)=q^2-5q+6$.  Considered as a set of vectors in $\Bbbk^4$, we can realize the arrangement as $x_0=e_1-e_2$, $x_1=e_1-e_3$, $x_2=e_1-e_4$, $x_3=e_2-e_3$, $x_4=e_2-e_4$, and $x_5=e_3-e_4$.  The circuits of this arrangement are:
\begin{multline*} \{x_0-x_1+x_3, x_1-x_2+x_5, x_0-x_2+x_4, x_3-x_4+x_5, \\ x_0-x_1+x_4-x_5, x_0-x_2+x_3+x_5, x_1-x_2-x_3+x_4 \}.
\end{multline*} 
The broken circuits are the sets $\{1, 3\}$, $\{2, 5\}$, $\{2, 4\}$,  $\{4, 5\}$, $\{ 1,4,5 \}$, $\{ 2,3,5 \}$ and $\{ 2,3,4 \}$. (The last three broken circuits contain other broken circuits and hence are non-minimal nonfaces.) The facets of the no broken circuit complex $\Delta^{\mathrm{NBC}}_{K_4}$ are $\{0,1, 2\}$, $\{0, 1, 4\}$, $\{0, 1, 5\}$, $\{0, 2, 3\}$, $\{0, 3, 4\}$, and $\{0, 3, 5\}$.  The number of facets should be the constant term of the characteristic polynomial, which is correct.  The facets of the reduced no broken circuit complex $\Delta^{RNBC}_{K_4}$ are $\{1, 2\}$, $\{1, 4\}$, $\{1, 5\}$, $\{2, 3\}$, $\{3, 4\}$, and $\{3, 5\}$, 
and the other faces are the empty set and the five vertices.

Let us take $\prec$ to be the natural order $\leq$.
The facets of the extended no broken circuit complex $\Delta^{\mathrm{ENBC}}_{K_4,\leq}$ are 
\begin{quote}
$\{y_0, x_2, x_4, x_5\}$, $\{y_0, y_1, x_4, x_5\}$, $\{y_0, y_2, x_4, x_5\}$, $\{y_0, y_3, x_2, x_5\}$,\\ 
$\{y_0, y_4, x_2, x_5\}$, $\{y_0, y_5, x_2, x_4\}$, $\{y_0, y_1, y_2, x_4\}$, $\{y_0, y_1, y_4, x_5\}$,\\
$\{y_0, y_1, y_5, x_4\}$, $\{y_0, y_2, y_3, x_5\}$, $\{y_0, y_3, y_4, x_2\}$, and $\{y_0, y_3, y_5, x_2\}$.  
\end{quote}
Hence the cohomology class of $V(\Delta^{\mathrm{ENBC}}_{K_4,\leq})$ in $H^*(\mathbb{P}^5\times\mathbb{P}^5)$ is
$$[V(\Delta^{\mathrm{ENBC}}_{K_4,\leq})]=
[\mathbb{P}^2\times\mathbb{P}^0]+5[\mathbb{P}^1\times\mathbb{P}^1]+6[\mathbb{P}^0\times\mathbb{P}^0].$$

The ideal $I(\Gamma_{K_4})$ can be presented, with each polynomial written in term order, as
\begin{align*}
I(\Gamma_{K_4})&=\langle
x_0-x_1+x_3, x_1-x_2+x_5, x_0-x_2+x_4, x_3-x_4+x_5,\\&\qquad
x_1y_1-x_0y_0, x_2y_2-x_0y_0, x_3y_3-x_0y_0, x_4y_4-x_0y_0, x_5y_5-x_0y_0,\\&\qquad
y_1y_3-y_0y_3+y_0y_1, y_2y_5-y_1y_5+y_1y_2, y_2y_4-y_0y_4+y_0y_2, \\&\qquad
    y_4y_5-y_3y_5+y_3y_4\rangle.
\end{align*}

The initial ideal $\init I(\Gamma_{K_4})$ is given by
$$\init I(\Gamma_{K_4})=
\langle x_0, x_1, x_3, x_2y_1, x_2y_2, x_4y_3, x_5y_1y_2, y_3y_4x_5,y_1y_3, y_2y_5, y_2y_4, y_4y_5\rangle.$$
For example, $x_2y_1\in \init I(\Gamma_{K_4})$ since $2$ is dependent on $\{1, 5\}\subseteq \{1, 3, 4, 5\}$.

The Hilbert series of $S/I(\Gamma_{K_4})$ is given by
\begin{align*}
H(\Gamma_{K_4}; q, t) 
&= \left(1+\frac{q}{1-q}\right)^3\left(1+\frac{t}{1-t}\right)
\\&\quad+5\left(1+\frac{q}{1-q}\right)^2\left(\frac{t}{1-t}\right)\left(1+\frac{t}{1-t}\right)
\\&\quad+6\left(1+\frac{q}{1-q}\right)\left(\frac{t}{1-t}\right)^2\left(1+\frac{t}{1-t}\right) \\
&= \left(\frac{t^3}{(1-qt)(t-1)^3}\right)
\\&\qquad\cdot
\left(\left(\frac{t-1}{(1-q)t}\right)^3-6\left(\frac{t-1}{(1-q)t}\right)^2+11\left(\frac{t-1}{(1-q)t}\right)-6\right).
\end{align*}

Setting $q=0$ gives
\begin{align*}
H(R_{K_4};t)
&= 1+6\left(\frac{t}{1-t}\right)+11\left(\frac{t}{1-t}\right)^2+6\left(\frac{t}{1-t}\right)^3 \\
&= \left(\frac{t^3}{(t-1)^3}\right)
\left(\left(\frac{t-1}{t}\right)^3-6\left(\frac{t-1}{t}\right)^2+11\left(\frac{t-1}{t}\right)-6\right).
\end{align*}

\end{document}